\documentclass{amsart}

\usepackage{amsmath,amssymb,amsthm} 
\usepackage{bbm} 
\usepackage{mathtools} 
\usepackage{hyperref} 
\usepackage{geometry} 
\usepackage{xcolor} 
\usepackage{xfp} 
\usepackage{ulem} 
\usepackage{todonotes} 
\usepackage{ifthen} 

\definecolor{theme}{RGB}{15, 87, 24} 
\definecolor{lighttheme}{RGB}{51, 153, 63} 
\definecolor{block}{RGB}{102, 60, 0} 
\definecolor{lightblock}{RGB}{255, 238, 214} 
\definecolor{alert}{RGB}{212, 43, 43} 

\newlength{\pagespace}
\newlength{\marginshift}
\hypersetup{
    colorlinks=true,
    linkcolor=theme,
    citecolor=lighttheme
}

\newtheorem{theorem}{Theorem}[section]
\newtheorem{proposition}[theorem]{Proposition}
\newtheorem{lemma}[theorem]{Lemma}
\newtheorem{corollary}[theorem]{Corollary}

\newcommand{\R}{\mathbb{R}}

\DeclarePairedDelimiterX\interval[1]{\{}{\}}{1,\dotsc,\ifthenelse{\equal{#1}{}}{n}{#1}}
\DeclarePairedDelimiterXPP{\e}[1]{\exp}{(}{)}{}{#1}
\let\P\relax
\let\O\relax
\let\o\relax
\DeclarePairedDelimiterXPP{\P}[1]{\mathbb{P}}{(}{)}{}{#1}
\DeclarePairedDelimiterXPP{\E}[1]{\mathbb{E}}{[}{]}{}{#1}
\DeclarePairedDelimiterXPP{\O}[1]{O}{(}{)}{}{#1}
\DeclarePairedDelimiterXPP{\o}[1]{o}{(}{)}{}{#1}
\newcommand{\I}[1]{\mathbbm{1}_{#1}}

\newcommand{\pts}{\mathcal{P}}
\newcommand{\inv}{\mathrm{Inv}}
\newcommand{\minv}{\mathcal{I}}
\newcommand{\rect}{\mathcal{T}}
\newcommand{\diag}{D}
\newcommand{\mass}{M}
\newcommand{\uniform}[1][{[0,1]}]{\textsc{Uniform}\left(#1\right)}
\newcommand{\ray}[1][2,1/(2\sigma^2)]{\textsc{Rayleigh}\left(#1\right)}

\newcommand{\poisson}[1][0,1]{\textsc{Poisson}\left(#1\right)}

\newcommand{\dconv}{\overset{\mathrm{d}}{\longrightarrow}}

\title{Minimal inversion of a permuton sample}
\author{Beno\^it Corsini}
\date{}

\begin{document}

\setcounter{tocdepth}{1}

\begin{abstract}
    Given a permutation $\sigma$, its corresponding \textit{inversion graph} is obtained by adding an edge between $i<j$ if and only if $\sigma(i)>\sigma(j)$.
    The first results on random inversion graphs come from Acan and Pittel~\cite{acan2013connected}, who studied the connected threshold for a uniform permutation with fixed inversion number, and Bhattacharya and Mukherjee~\cite{bhattacharya2017degree}, who mostly focused on the degrees of the graph when the permutation is chosen uniformly at random.
    
    In this work, we call \textit{minimal inversion} the minimal degree of the inversion graph and extend a theorem from Bhattacharya and Mukherjee~\cite[Theorem~3.5]{bhattacharya2017degree} to the case where the permutation is not only uniform, but obtained as the ordering of points sampled according to some distribution on the plane.
    Under regularity assumptions on the distribution, and for the appropriate $\alpha>0$, we show that the probability that the minimal inversion rescaled by $n^{\alpha/(\alpha+1)}$ is larger than $t$ behaves like $\exp(-ct^{\alpha+1})$ for some constant $c>0$ depending on the distribution.
    We further show that every $\alpha>0$ admits at least one corresponding distribution, thus proving that the minimal inversion can asymptotically scale as $n^\beta$ for any $\beta\in[0,1]$ (the cases $\beta=0$ and $\beta=1$ being obtained via the identity and anti-identity permutations, among others).
\end{abstract}

\maketitle

\tableofcontents

\section{Introduction}

The \textit{inversion graph} of a permutation $\sigma$ is the structure obtained by connecting any pair $\{i<j\}$ together if they form an inversion: $\sigma(i)>\sigma(j)$.
This model was first introduced to study structural properties of the corresponding permutation~\cite{pnueli1971transitive}, with a strong emphasis on the uniqueness of the connected component~\cite{acan2013connected,koh2007connected}.
Extending upon this work, Bhattacharya and Mukherjee~\cite{bhattacharya2017degree} studied several properties of the degrees of the graph, focusing on the case of a uniformly chosen permutation (although they sometimes allow broader assumptions).

When working on expanding results from uniform permutations to more generic models, it is common to either consider a parametrized model to which the uniform is a special case, models such as Mallows~\cite{mallows1957non}, Ewens~\cite{ewens1972sampling}, or more recently record-biased~\cite{bouvel2026record} permutations, or to consider restricted uniform permutations, usually via pattern-avoidance~\cite{simion1985restricted}.
However, with the introduction of permutons~\cite{hoppen2013limits}, these two approaches can now often fall under the same methodology.
Indeed, Mallows~\cite{starr2009thermodynamic}, Ewens~\cite{feray2013asymptotic}, record-biased~\cite{bouvel2026record}, and pattern-avoiding~\cite{borga2023skew} permutations are all known to converge to permutons in most cases.

The strength of permutons, however, is also their limitation.
Indeed, when studying a specific statistic, the infinite diversity of possible distributions potentially leads to erratic behaviour and thus often requires smoothness assumptions.
For example, the cycle structure requires an equi-continuous density~\cite{mukherjee2016fixed}, the longest increasing subsequence only applies when it is of linear order~\cite{dubach2024increasing}, and the structure of the binary search tree relies on a density close to the left boundary~\cite{corsini2025binary}.

In this work, the statistic we are interested in is the \textit{minimal inversion}, which corresponds to the minimal degree of the inversion graph.
We extend the result from Bhattacharya and Mukherjee~\cite[Theorem~3.5]{bhattacharya2017degree}, which states that the minimal inversion of a uniform permutation divided by $\sqrt{n}$ converges to a Rayleigh distribution with parameter $1/\sqrt{2}$.
In our setting, there exist two parameters $\alpha>0$ and $c>0$ depending on the distribution of the permuton such that the minimal inversion divided by $n^{\alpha/(\alpha+1)}$ has asymptotic tail probabilities equal to $\e{-ct^{\alpha+1}}$ (see Theorem~\ref{thm:main} and Theorem~\ref{thm:tile});
in the case of a uniform permutation, $\alpha=1$ and $c=1$, which exactly corresponds to~\cite[Theorem~3.5]{bhattacharya2017degree}.

Similar to the previous statistics (cycles, increasing subsequences, and binary search trees), our results require some smoothness assumption, namely the existence of a density in a neighbourhood of $(0,0)$ and $(1,1)$ along with a non-zero mass in the top-left and bottom-right corners of the unit square (see Section~\ref{sec:ass}).
Since the corner assumptions are often too restrictive, we further extend our results to permutons which can be decomposed into sub-permutons satisfying the previous assumptions (see Section~\ref{sec:tile});
for example, this generalization allows us to apply our results to permutons obtained as the marginal distribution of a fixed-size permutation (see Corollary~\ref{cor:permutation}).

\subsection{Fundamental definitions}\label{sec:def}

Let $\pts=\{(X_i,Y_i):1\leq i\leq n\}$ be an arbitrary set of points with no two identical coordinates.
For any index $i\in[n]=\interval{}$, we denote by $\inv_\pts(i)$ the set of (left and right) inversions created by the $i$-th point, that is
\begin{align}\label{eq:inv}
    \inv_\pts(i)
    =\Big\{j:(X_i-X_j)(Y_i-Y_j)<0\Big\}\,.
\end{align}
Similarly, if $\sigma$ is a permutation of $[n]$, then we define its set of inversion as $\inv_\sigma=\inv_{\pts_\sigma}$ where $\pts_\sigma=\{(i,\sigma(i)):1\leq i\leq n\}$.
It is worth mentioning that, in the context of permutations, the definition of $\inv$ is not completely standard as it usually refers to the total number of left (or right) inversions, and satisfies
\begin{align*}
    \big|\big\{i<j:\sigma(i)>\sigma(j)\big\}\big|
    =\frac{1}{2}\sum_{i\in[n]}\big|\inv_\sigma(i)\big|\,.
\end{align*}
Finally, given a set of points $\pts$, we call \textit{minimal inversion} the minimal size of an inversion set, that is the minimal value of $|\inv_\pts(i)|$ over $i\in[n]$.

A \textit{permuton} is a probability measure $\mu$ on $[0,1]^2$ with uniform marginals.
That is, for any measurable set $A$, we have $\mu([0,1]\times A)=\mu(A\times[0,1])=\lambda_1(A)$, where $\lambda_1$ is the standard Lebesgue measure on $\R$.
The main goal of this article is to study the distribution of the minimal inversion obtained from a set of points sampled according to a permuton.
More precisely, for any permuton $\mu$, we let $\minv_n^\mu$ be the random variable defined by
\begin{align}\label{eq:main}
    \minv_n^\mu
    =\min\Big\{\big|\inv_\pts(i)\big|:i\in[n]\Big\}\,,
\end{align}
where $\pts$ are $n$ points sampled according to $\mu$.
Before stating our main results, we need to state the required assumptions on $\mu$.

\subsection{Assumptions}\label{sec:ass}

We say that $\mu$ is \textit{BL-smooth} (for \textit{Bottom-Left}) if there exists $0<\epsilon\leq1/2$ such that $\mu$ admits a density $f$ on $[0,\epsilon]^2$ satisfying the following properties.
There exists $\alpha>0$ and a measurable function $\phi:[0,\infty)^2\rightarrow[0,\infty)$ bounded on every compact and strictly positive almost-everywhere such that $f(tx,ty)\sim t^{\alpha-1}\phi(x,y)$ when $t\rightarrow0$.
We then say that $\mu$ is \textit{TR-smooth} (for \textit{Top-Right}) if its symmetric distribution with respect to $y=1-x$ is BL-smooth.

Note that if $\mu$ is both BL and TR-smooth, then the values of $\alpha$ and $\phi$ do not need to be identical between the BL and TR parts.
We also mention here that not any $\phi$ can satisfy such assumption, since this definition implies that $\phi(tx,ty)=t^{\alpha-1}\phi(x,y)$;
a simple example of such a function is $\phi(x,y)=(x+y)^{\alpha-1}$, which is fundamental for the example from Section~\ref{sec:standard example}.

For any $\alpha>0$ and $\phi$ as above, we say that $\mu$ is \textit{$(\alpha,\phi)$-smooth} if it is both BL and TR-smooth, with its density function satisfying $f(tx,ty)+f(1-tx,1-ty)\sim t^{\alpha-1}\phi(x,y)$.
Observe that, by letting $(\alpha_-,\phi_-)$ and $(\alpha_+,\phi_+)$ be the parameters respectively for the BL and TR smoothness of $\mu$, then we have that $\alpha=\min\{\alpha_-,\alpha_+\}$ and $\phi=\phi_-\I{\alpha_-\leq\alpha_+}+\phi_+\I{\alpha_+\leq\alpha_-}$.

The BL-smoothness of a permuton allows us to control the degrees of the nodes corresponding to points within $[0,\epsilon]^2$ (see Proposition~\ref{prop:corner inv});
equivalently by symmetry, the TR-smoothness controls the degrees of the points within $[1-\epsilon,1]^2$.
However, this does not necessarily prevent smaller degrees from existing in the graph (see Section~\ref{sec:counter}).
In order to prevent small degrees elsewhere in the graph, we need an extra assumption on $\mu$.

We say that $\mu$ is \textit{corner-balanced} if, for any $0<\epsilon\leq1/2$, there exists $0<\rho<1$ such that
\begin{align*}
    \min\Big\{\mu\big([0,\epsilon]\times[\rho,1]\big),\mu\big([1-\epsilon,1]\times[0,\rho]\big)\Big\}>0\,.
\end{align*}
This assumption implies that, for any $\epsilon$, the degrees of the nodes corresponding to points with $\epsilon\leq X_i\leq1-\epsilon$ have degree at least $n$ (see Lemma~\ref{lem:top-bottom}).
This, however, needs to be combined with the smoothness assumption to be able to fully control the minimum degree (see Lemma~\ref{lem:top-bottom} and Lemma~\ref{lem:left}).

\medskip

We conclude this section by saying that a permuton $\mu$ is \textit{$(\alpha,\phi)$-standard} if it is $(\alpha,\phi)$-smooth and corner-balanced.
Using the previous definitions, a permuton is $(\alpha,\phi)$-standard if and only if there exists $f:[0,1]^2\rightarrow[0,\infty)$, $\alpha_\pm>0$, and $\phi_\pm:[0,\infty)^2\rightarrow[0,\infty)$ bounded on every compact and strictly positive almost-everywhere such that, for any $\epsilon>0$ small enough, it satisfies the following properties.
\begin{itemize}
    \item $\mu$ has density $f$ on $[0,\epsilon]^2$ and on $[1-\epsilon]^2$.
    \item $f(tx,ty)\sim t^{\alpha_--1}\phi_-(x,y)$ and $f(1-tx,1-ty)\sim t^{\alpha_+-1}\phi_+(x,y)$ when $t\rightarrow0$.
    \item $f(tx,ty)+f(1-tx,1-ty)\sim t^{\alpha-1}\phi(x,y)$ (equivalently $\alpha=\min\{\alpha_-,\alpha_+\}$ and $\phi=\phi_-\I{\alpha_-\leq\alpha_+}+\phi_+\I{\alpha_+\leq\alpha_-}$).
    \item $\mu([0,\epsilon]\times[\rho,1])>0$ and $\mu([1-\epsilon,1]\times[0,\rho])>0$ for some $0<\rho<1$.
\end{itemize}
We now show the asymptotic distribution of the minimal inversion of a standard permuton.

\subsection{Main result}

We denote by $\ray[a,b]$ the distribution of the random variable $R$ whose support is $[0,\infty)$ and such that, for any $t\geq0$, we have $\P{R>t}=e^{-bt^a}$.
Note that this is a generalization of the standard Rayleigh distribution, corresponding to $\ray$.
By convention, the $\ray[\infty,b]$ distribution is simply equal to $b$ with probability $1$.

\begin{theorem}\label{thm:main}
    Let $\mu$ be a $(\alpha,\phi)$-standard permuton as defined in Section~\ref{sec:ass}.
    Let $\minv_n^\mu$ be the minimal inversion of $n$ points distributed according to $\mu$, as defined in~\eqref{eq:main}.
    Finally, denote by $\Delta=\{(x,y)\in[0,1]^2:x+y\leq1\}$ the unit square triangle.
    Then, the random variable $\minv_n^\mu$ satisfies
    \begin{align*}
        n^{-\alpha/(\alpha+1)}\minv_n^\mu
        \dconv\ray[\alpha+1,\int_\Delta\phi]
    \end{align*}
    as $n$ goes to infinity, where the convergence occurs in distribution.
\end{theorem}

The proof of Theorem~\ref{thm:main} can be found in Section~\ref{sec:main}.
Combining this result with the fact that any $\alpha>0$ admits an acceptable $\mu$, as shown in Section~\ref{sec:standard example}, we see that the minimal inversion can scale as $n^\beta$ for any $0<\beta<1$.
Moreover, one can easily generate permutons $\mu$ which scale linearly or of constant order (for example the identity and anti-identity permuton, but also examples from Section~\ref{sec:counter}), so all scales from constant to linear order are possible for the minimal inversion of a permuton sample.

The assumptions of Theorem~\ref{thm:main} are based on technical properties required for the convergence of the minimal degree to hold.
However, most common permutons either do not satisfy the standard assumption (usually because they do not admit a density), or have a piece-wise continuous and strictly positive density, which can also be chosen to be continuous at $(0,0)$ and $(1,1)$.
We thus provide here the simpler statement of Theorem~\ref{thm:main} in the latter case.

\begin{corollary}\label{cor:main}
    Let $\mu$ be a permuton which admits a piece-wise continuous and strictly positive density $f:[0,1]^2\rightarrow(0,\infty)$ continuous at $(0,0)$ and at $(1,1)$.
    Let $\minv_n^\mu$ be the minimal inversion of $n$ points distributed according to $\mu$, as defined in~\eqref{eq:main}.
    Then, the random variable $\minv_n^\mu$ satisfies
    \begin{align*}
        n^{-1/2}\minv_n^\mu
        \dconv\ray[2,\frac{f(0,0)+f(1,1)}{2}]
    \end{align*}
    as $n$ goes to infinity, where the convergence occurs in distribution.
\end{corollary}

\begin{proof}
    To prove this theorem, we simply need to observe that $\mu$ is $(1,\phi)$-standard with $\phi$ constant equal to $f(0,0)+f(1,1)$.
    Indeed, we see that $\mu$ is BL and TR-smooth since it admits $f$ as a density on $[0,\epsilon]^2$ and $[1-\epsilon,1]^2$ and $f(tx,ty)\sim f(0,0)$ and $f(1-tx,1-ty)\sim f(1,1)$ when $t\rightarrow0$.
    Furthermore $\mu$ is corner-balanced since $f$ is strictly positive and so $\mu([0,\epsilon]\times[1/2,1])>0$ and $\mu([1-\epsilon,1]\times[0,1/2])>0$.
    It follows that $\mu$ is $(1,\phi)$-standard and we can apply Theorem~\ref{thm:main}.
    The desired result follows from the fact that $\phi$ is constant equal to $f(0,0)+f(1,1)$ and the unit square triangle $\Delta$ has measure $1/2$.
\end{proof}

We observe that Corollary~\ref{cor:main} is a direct extension of~\cite[Theorem~3.5]{bhattacharya2017degree} stating that the limit of the minimum inversion on the uniform permuton divided by $\sqrt{n}$ is $\ray[2,1]$.

While the smoothness assumption plays a key role in the asymptotic behaviour of the minimal inversion, the corner-balanced property may seem unnecessary as several natural models of permutons do not satisfy it, yet might still have a similar limiting distribution (consider for example the uniform permuton with support on $[0,1/2]^2$ and $[1/2,1]^2$).
To circumvent this issue, we now extend the results of Theorem~\ref{thm:main} to permutons which can be decomposed into several squares on which the corresponding distribution corresponds to a standard permuton.

\subsection{Extension to tile-standard permutons}\label{sec:tile}

We call \textit{tiling} a family of closed and non-empty interior squares $\rect=\{S_1,\ldots,S_k\}$ such that their horizontal and vertical intervals cover the whole segment $[0,1]$ and their interior do not intersect.
In other words, the horizontal and vertical intervals of $S_1,\ldots,S_k$ almost form a partition, except that they are all closed so some of them will intersect on their boundary.

Given a tiling $\rect=\{S_1,\ldots,S_k\}$, we can define an ordering on their $x$ coordinates by letting
\begin{align*}
    S_\ell\prec_xS_m
    ~\Longleftrightarrow~
    \max\big\{x:(x,y)\in S_\ell\big\}\leq\min\big\{x:(x,y)\in S_m\big\}\,.
\end{align*}
Since the sides of the squares are intervals, this relationship is well defined and any pair $(S_\ell,S_m)$ with $\ell\neq m$ satisfies either $S_\ell\prec_xS_m$ or $S_m\prec_xS_\ell$.
We similarly defined $\prec_y$ by replacing $x$ in both maximum and minimum with $y$.
We now always assume that a tiling $\rect=\{S_1,\ldots,S_k\}$ is defined so that $S_1\prec_x\cdots\prec_xS_k$ and let $\sigma_\rect$ be the unique permutation such that
\begin{align}\label{eq:rect perm}
    \sigma_\rect(\ell)<\sigma_\rect(m)
    ~\Longleftrightarrow~
    S_\ell\prec_yS_m\,;
\end{align}
we call $\sigma_\rect$ the \textit{permutation corresponding to the tiling}.

Given a tiling $\rect=\{S_1,\ldots,S_k\}$ and their corresponding permutation $\sigma=\sigma_\rect$, we say that $\ell$ is \textit{diagonally splitting} (or simply \textit{splitting}) if $\inv_\sigma(\ell)=\emptyset$;
in other words, it is splitting if there are no other squares to its upper-left or bottom-right side.
We denote by $\diag_\rect$ the set of splitting indices.
Finally, for any $\ell$, denote by
\begin{align*}
    L_\rect(\ell)
    =\sum_{m\in\inv_\sigma(\ell)}\sqrt{\lambda_2(S_m)}
\end{align*}
the total length of the squares creating inversions with $\ell$, where $\lambda_2$ is the standard Lebesque measure on $\R^2$ (so that $\sqrt{\lambda_2(S_m)}$ is the side length of $S_m$) and $\inv_\sigma$ is the set of inversions as defined in~\eqref{eq:inv}.
Extend this definition to $\mass_\rect$, the \textit{minimal anti-diagonal mass}, defined by
\begin{align}\label{eq:mass}
    \mass_\rect
    =\min\Big\{L_\rect(\ell):\ell\in[k]\Big\}\,.
\end{align}
Using the convention that an empty sum equals $0$, we see that $\mass_\rect=0$ if and only if $D_\rect$ is not empty, but also if and only if the minimal inversion of $\sigma$ is equal to $0$.

\medskip

Given a closed and non-empty interior square $S=[x_0,x_0+\delta]\times[y_0,y_0+\delta]\subseteq[0,1]^2$, we call \textit{sub-permuton} a distribution $\mu_S$ on $S$ with Lebesgue-distributed marginals.
Any sub-permuton can be linearly transformed into a permuton on $[0,1]$;
for example, if $\mu_S$ has a density $f_S$ on $S$, then the shifted density $\Tilde{f}_S(x,y)=\delta f(x_0+\delta x,y_0+\delta y)$ corresponds to a permuton on $[0,1]^2$.
We thus say that $\mu_S$ is $(\alpha,\phi)$-standard if its rescaled permuton is $(\alpha,\delta^\alpha\phi)$-standard.
Observe that the definition of standard is coherent since a permuton is actually also a sub-permuton with $\delta=1$.

A permuton $\mu$ is now said to be $(\rect,\alpha,\phi)$-tile-standard is it satisfies the following properties.
First, $\rect=\{S_1,\ldots,S_k\}$ is a tiling such that
\begin{align*}
    \mu\Big(S_1\cup\cdots\cup S_k\Big)
    =1\,.
\end{align*}
Second, for any $\ell\in[k]$, the restriction $\mu_\ell$ of $\mu$ to $S_\ell$ is a $(\alpha_\ell,\phi_\ell)$-standard permuton, for some $\alpha_\ell>0$ and $\phi_\ell$.
Third and final, the pair $(\alpha,\phi)$ satisfies
\begin{align*}
    \alpha
    =\min\Big\{\alpha_\ell:\ell\in\diag_\rect\Big\}
\end{align*}
and
\begin{align*}
    \phi
    =\sum_{\ell\in\diag_\rect}\phi_\ell\I{\alpha_\ell=\alpha}\,,
\end{align*}
where $\diag_\rect$ is the set of splittings.
In the case where $\diag_\rect=\emptyset$, we let $\alpha=\infty$ and $\phi=0$.

It is worth mentioning that the parameters $(\rect,\alpha,\phi)$ involved in the definition of tile-standard permutons are unique (see Proposition~\ref{prop:unique tile}).
This comes from the fact that standard permutons cannot be decomposed into non-trivial tilings, since they need non-zero density close to their bottom-left and top-right corners (smooth) and non-zero mass towards their top-left and bottom-right corners (corner-balanced).
We also mention here that any $(\alpha,\phi)$-standard permuton is also a $(\{[0,1]^2\},\alpha,\phi)$-tile-standard permuton.
We now state the convergence of the minimal inversion for a tile-standard permuton.

\begin{theorem}\label{thm:tile}
    Let $\mu$ be a $(\rect,\alpha,\phi)$-tile-standard permuton as defined in Section~\ref{sec:tile}.
    Let $\minv_n^\mu$ be the minimal inversion of $n$ points distributed according to $\mu$, as defined in~\eqref{eq:main}.
    Finally, denote by $\Delta=\{(x,y)\in[0,1]^2:x+y\leq1\}$ the unit square triangle.
    Then, the random variable $\minv_n^\mu$ satisfies
    \begin{align*}
        n^{-\alpha/(\alpha+1)}\minv_n^\mu
        \dconv\ray[\alpha+1,\mass_\rect+\int_\Delta\phi]
    \end{align*}
    as $n$ goes to infinity, where $\mass_\rect$ is defined in~\eqref{eq:mass} and the convergence occurs in distribution.
\end{theorem}

The proof of Theorem~\ref{thm:tile} can be found in Section~\ref{sec:tile result}.
The strategy to prove this result is actually to split the limit according to two cases: whether $\diag_\rect$ is empty or not.
\begin{itemize}
    \item If the set of splittings $\diag_\rect$ is non empty, then $\alpha\in(0,\infty)$ and $\mass_\rect=0$, so the statement of the theorem is identical to that of Theorem~\ref{thm:main}, except that the definition of $\phi$ takes into account the tiling structure of $\mu$.
    \item If the set of splittings $\diag_\rect$ is empty, then $\alpha=\infty$, $\mass_\rect>0$, and the limiting distribution is deterministic, so the convergence rewrites as $n^{-1}\minv_n^\mu\rightarrow\mass_\rect>0$.
\end{itemize}
Using that any standard permuton is also tile-standard, Theorem~\ref{thm:tile} strictly extends Theorem~\ref{thm:main}.
However, we keep the two results separated for two reasons: since most non-trivial permutons are standard, and since the definitions involved in tile-standard are a lot more involved.
We now further provide extensions of Theorem~\ref{thm:tile} for when the decompositions are not only standards, but admit a piece-wise continuous and strictly positive density.

\begin{corollary}\label{cor:tile}
    Let $\mu$ be a permuton decomposed into $\mu_1,\ldots,\mu_k$ with respect to the tiling $\rect=\{S_1,\ldots,S_k\}$ with side lengths $\delta_1,\ldots,\delta_k$ and bottom-left corners $(x_1,y_1),\ldots,(x_k,y_k)$.
    Assume that each measure $\mu_\ell$ admits a piece-wise continuous, strictly positive density $f_\ell$ on $S_\ell$ which is also top-right continuous at $(x_\ell,y_\ell)$ and bottom-left continuous at $(x_\ell+\delta_\ell,y_\ell+\delta_\ell)$.
    Let $\minv_n^\mu$ be the minimal inversion of $n$ points distributed according to $\mu$, as defined in~\eqref{eq:main}.
    Then, the random variable $\minv_n^\mu$ satisfies one of the following two cases, as $n\rightarrow\infty$.
    \begin{itemize}
        \item If the set of splittings $\diag_\rect$ is non empty, then
        \begin{align*}
            n^{-1/2}\minv_n^\mu
            \dconv\ray[2,\sum_{\ell\in\diag_\rect}\frac{f_\ell(x_\ell,y_\ell)+f_\ell(x_\ell+\delta_\ell,y_\ell+\delta_\ell)}{2}]
        \end{align*}
        \item If the set of splittings $\diag_\rect$ is empty, then
        \begin{align*}
            n^{-1}\minv_n^\mu
            \dconv\mass_\rect
            =\min\left\{\sum_{m\in\inv_{\sigma_\rect}(\ell)}s_m:\ell\in[k]\right\}\,.
        \end{align*}
    \end{itemize}
    Both convergence occur in distribution, although the second one also occurs in probability.
\end{corollary}

\begin{proof}
    We follow the same strategy as for the proof of Corollary~\ref{cor:main} and show that $\mu$ is $(\rect,1,\phi)$-tile-standard for the appropriate $\phi$.
    We denote by $\phi_\ell=f_\ell(x_\ell,y_\ell)+f_\ell(x_\ell+\delta_\ell,y_\ell+\delta_\ell)$.
    First, the definition of $\rect$ directly tells us that
    \begin{align*}
        \mu\Big(S_1\cup\ldots\cup S_k\Big)
        =1\,.
    \end{align*}
    Second, the linear transformation of $f_\ell$ into the density of a standard permuton is given by $\Tilde{f}_\ell(x,y)=\delta_\ell f_\ell(x_\ell+\delta_\ell x,y_\ell+\delta_\ell y)$ so that
    \begin{align*}
        \Tilde{f}_\ell(tx,ty)+\Tilde{f}_\ell(1-tx,1-ty)
        \sim\delta_\ell\phi_\ell\,,
    \end{align*}
    telling us that $\Tilde{f}_\ell$ is $(1,\delta_\ell\phi_\ell)$-standard so that $f_\ell$ is $(1,\phi_\ell)$-standard.
    Finally, we see that all $\alpha_\ell$ are equal to $1$ so that $\alpha=1$ (or $\alpha=\infty$ when $\diag_\rect$ is empty) and
    \begin{align*}
        \phi
        =\sum_{\ell\in\diag_\rect}\phi_\ell\,.
    \end{align*}
    The result of the corollary is then a direct application of Theorem~\ref{thm:tile}, realizing that the definition of $\mass_\rect$ is correct since $\sqrt{\lambda_2(S_\ell)}=\delta_\ell$.
\end{proof}

The statement of Corollary~\ref{cor:tile} simplifies some of the assumptions from Theorem~\ref{thm:tile} but still applies to a wide range of permutons.
We conclude this section by providing the limit of the minimal inversion when the permuton is obtained from a finite permutation.
For a permutation $\sigma$ of $[k]$, we denote by $\mu_\sigma$ the corresponding \textit{empirical} permuton whose density is given by
\begin{align*}
    f(x,y)
    =\left\{\begin{array}{ll}
        k & \textrm{if $\lceil ky\rceil=\sigma(\lceil kx\rceil)$} \\
        0 & \textrm{otherwise}\,.
    \end{array}\right.
\end{align*}

\begin{corollary}\label{cor:permutation}
    Let $\sigma$ be a permutation of $[k]$ and $\mu=\mu_\sigma$ be the corresponding empirical permuton;
    denote by $\diag_\sigma$ the set of splittings of $\sigma$ (relative to its natural stile decomposition).
    Let $\minv_n^\mu$ be the minimal inversion of $n$ points distributed according to $\mu$, as defined in~\eqref{eq:main}.
    Then, the random variable $\minv_n^\mu$ satisfies one of the following two cases, as $n\rightarrow\infty$.
    \begin{itemize}
        \item If the set of splittings $\diag_\sigma$ is non empty, then
        \begin{align*}
            n^{-1/2}\minv_n^\mu
            \dconv\ray[2,\Big.k\big|\diag_\sigma\big|]
        \end{align*}
        In particular, when $\sigma$ is the identity, we have
        \begin{align*}
            n^{-1/2}\minv_n^\mu
            \dconv\ray[2,\big.k^2]
        \end{align*}
        \item If the set of splittings $\diag_\sigma$ is empty, then
        \begin{align*}
            n^{-1}\minv_n^\mu
            \dconv\frac{1}{k}\min\Big\{\big|\inv_\sigma(\ell)\big|:\ell\in[k]\Big\}\,.
        \end{align*}
        In particular, when $\sigma$ is the anti-identity, we have
        \begin{align*}
            n^{-1}\minv_n^\mu
            \dconv\frac{k-1}{k}\,.
        \end{align*}
    \end{itemize}
    Both convergence occur in distribution, although the second one also occurs in probability.
\end{corollary}

\begin{proof}
    This is a direct application of Corollary~\ref{cor:tile} using that, when $\mu$ is the empirical permuton of some $\sigma$, then $\delta_\ell=1/k$ and $f_\ell=k$ for all $\ell\in[k]$.
\end{proof}

It is interesting to observe that, the limits when $k\rightarrow\infty$ for the identity and anti-identity are coherent with the minimal inversion for the identity and anti-identity permutation.
Indeed, the identity permutation has minimal inversion equal to $1$ and, as $k\rightarrow\infty$, the distribution of $\ray[2,k^2]$ converges to a determnistic measure with probability $1$ on $0$ (so $\minv=o(n^{1/2})$).
For the anti-identity permutation, the minimal inversion equals $n-1$ and we see that, as $k\rightarrow\infty$, the result of Corollary~\ref{cor:permutation} states that $n^{-1}\minv\rightarrow1$, as expected.

Before moving into the background necessary for the different proofs, it is worth discussing the possible fluctuations of $\minv_n^\mu$ when $\diag_\rect$ is empty.
As one can see from the proof of Theorem~\ref{thm:tile} (see Proposition~\ref{prop:non-diag}), the leading term when $\diag_\rect=\emptyset$ arises from the law of large number applied to the number of points falling within each square $S_1,\ldots,S_k$.
With this approach, one might expect that $\minv_n^\mu$ satisfies a standard central limit theorem, arising from the fluctuation in the number of points within each $S_1,\ldots,S_k$.
However, this is only one side of the behaviour of the minimal degree, as there are also inversions created within each square and these follow the distribution stated by Theorem~\ref{thm:main}, thus of order $n^{\alpha_\ell/(\alpha_\ell+1)}$.
This means that it is reasonable to expect the second order fluctuations of $\minv_n^\mu$ to be either distributed as Rayleigh when there exists some $\ell$ which minimizes $\mass_\rect$ and such that $\alpha_\ell>1$ (leading to $\alpha_\ell/(\alpha_\ell+1)>1/2$), as a normal random variable when all such $\alpha_\ell$ are strictly less than $1$, and as a combination of both when $\alpha_\ell=1$.
Since our focus was on standard permutons and sub-linear minimal inversions, we do not provide the detailed distributional limit when $\diag_\rect=\emptyset$ and let these detailed fluctuations as a consequence of the results above.

\subsection{Notations and general assumptions}\label{sec:notation}

Throughout this work, for any $x_n,y_n$, we use the notation $x_n=o(y_n)$ to say that $x_n/y_n\rightarrow0$ and $x_n\sim y_n$ to say that $x_n/y_n\rightarrow1$.
For any sequence $x_n$, we also say that it is slowly varying when $x_{cn}\sim x_n$ for any $c>0$.
From now on, we also drop the superscript $\mu$ when the dependency is clear from context and let $\pts=\pts_n=\{(X_1,Y_1),\ldots,(X_n,Y_n)\}$ be $n$ points sampled according to $\mu$.

Since most of the convergence results occur in distribution, we mostly consider this topology and simply recall that, when the limit is a constant, the convergence in distribution is equivalent to the convergence in probability.
We further mention here that some results likely also occur almost-surely (such as Lemma~\ref{lem:strip bin} or Proposition~\ref{prop:non-diag}).
However, since they are only used to prove distributional limits, we prefer their simpler statements over more complete but also more complex proofs.

For any set $S\subseteq[0,1]^2$, we denote by $S^c=[0,1]^2\setminus S$ its complement with respect to $[0,1]^2$ and $S^{-1}=\{(1-x,1-y):(x,y)\in S\}$ its symmetric set with respect to $y=1-x$.
Given $n$ points sampled according to $\mu$, we further let
\begin{align*}
    \minv_n(S)
    =\min\Big\{\big|\inv_\pts(i)\big|:(X_i,Y_i)\in S\Big\}
\end{align*}
be the minimal inversion restricted to the set $S$.
Finally, for any $0<\epsilon\leq1/2$, we denote by
\begin{align*}
    S_\epsilon
    =\big\{i:0\leq X_i,Y_i\leq\epsilon\big\}
\end{align*}
the set of points within $[0,\epsilon]^2$.
Observe that $S_\epsilon^{-1}$ is simply the set of points within the interval $[1-\epsilon]^2$.

Before moving onto the proofs of the two theorems, we provide a few useful notations.
Given $\alpha>0$, we let $\beta=\alpha/(\alpha+1)$ and $\gamma=\beta/\alpha=1/(\alpha+1)$;
we note that $1-\gamma=\beta$.
Moreover, if we consider a variety of $\alpha$, such as $\alpha_1,\ldots,\alpha_k$, then we denote by $\beta_1,\ldots,\beta_k$ and $\gamma_1,\ldots,\gamma_k$ the corresponding parameters.

\section{BL-smooth permutons}\label{sec:bl}

In this section, we focus on a BL-smooth permuton $\mu$ and on the minimal inversion of the bottom-left nodes, corresponding to $\minv_n(S_\epsilon)$ as defined in Section~\ref{sec:notation}.
We also rely on the definitions of $\beta$ and $\gamma$ from Section~\ref{sec:notation}.

For the rest of this section, we let $\epsilon=\epsilon_n$ be slowly varying in $n$ and converging towards $0$; in particular, this implies that $\epsilon n$ is of larger order than $n^\beta$ for any finite $\alpha>0$.
Thanks to the BL-smoothness of $\mu$, we know that $f(tx,ty)\sim t^{\alpha-1}\phi(x,y)$ with $\phi$ bounded on any compact and strictly positive almost-everywhere.
We denote by $\|\phi\|$ the maximal value of $\phi$ on $[0,1]^2$ and observe that, for any $0\leq a<b$, $0\leq c<d$, $\zeta\geq\max\{b,c\}$, and $\delta=\delta_n=o(\epsilon)$, we have
\begin{align}\label{eq:f int}
    \int_{a\delta}^{b\delta}\int_{c\delta}^{d\delta}f(x,y)dydx
    =(\zeta\delta)^2\int_{a/\zeta}^{b/\zeta}\int_{c/\zeta}^{d/\zeta}f(\delta\zeta x,\delta\zeta y)dydx
    \sim(\zeta\delta)^{\alpha+1}\int_{a/\zeta}^{b/\zeta}\int_{c/\zeta}^{d/\zeta}\phi\,.
\end{align}
This behaviour will be useful when integrating $f$ over intervals of smaller order than $\epsilon$.

The goal of this section is to prove that $\minv_n(\epsilon)$ rescaled by $n^\beta$ converges in distribution to the generalized Rayleigh distribution.
We consider the following sets, extending the notations from Section~\ref{sec:notation}.
For any $c\geq0$, we let $S_c^x=\{i:X_i\leq cn^{-\gamma}\}$ (respectively $S_c^y=\{i:Y_i\leq cn^{-\gamma}\}$) be the set of points in the leftmost (respectively bottommost) band of length $cn^{-\gamma}$.
Then, for any $0\leq a\leq b$, we let
\begin{align*}
    S_{a,b}
    =\big(S_b^x\setminus S_{a-}^x\big)\cap\big(S_b^y\setminus S_{a-}^y\big)
    =\Big\{i:an^{-\gamma}\leq X_i,Y_i\leq bn^{-\gamma}\Big\}
\end{align*}
be the set of points within the interval $[a,b]$ when rescaled by $n^{-\gamma}$.
For the rest of this section, $b$ will be constant but chosen large enough (usually such that $b>t$ for some well-defined $t$), $a=a_n$ will be slowly varying and converging to $0$, so that $cn^\beta$ diverges to infinity for any $c\in[a,b]$.
We provide in Figure~\ref{fig:corner} a representation of the different sets, to better understand their role and relationships in terms of inversions.

\begin{figure}[htb]
    \centering
    \includegraphics[width=0.8\textwidth]{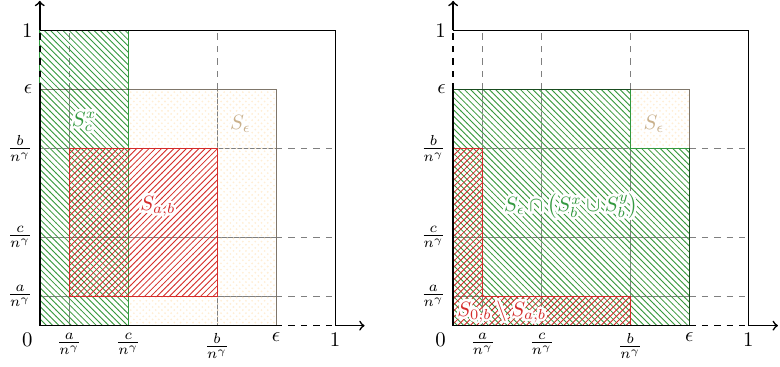}
    \caption{
    A representation of the different sets introduced in Section~\ref{sec:notation} and in Section~\ref{sec:bl}, along some sets used in Section~\ref{sec:binom}.
    On the left figure:
    the set in light brown is $S_\epsilon$ and correspond to points whose coordinates belong to $[0,\epsilon]^2$;
    the set in green is $S_c^x$ and corresponds to points whose $x$ coordinates are less than $cn^{-\gamma}$; and
    the set in red is $S_{a,b}$ and corresponds to points whose coordinates are between $an^{-\gamma}$ and $bn^{-\gamma}$.
    On the right figure:
    the set in red is used in Lemma~\ref{lem:corner bin} and is empty with high probability; and
    the set in green is used in Lemma~\ref{lem:L bin} and will contain $o(n^\beta)=o(n^{1-\gamma})$ points.
    }
    \label{fig:corner}
\end{figure}

The rest of the section is organized as follows.
In Section~\ref{sec:binom}, we provide some asymptotic results regarding the size of the different sets introduced above.
In Section~\ref{sec:useful}, we prove a key result regarding points of $S_{a,b}$, which is at the heart of why the limit in Theorem~\ref{thm:main} is a generalized Rayleigh.
Finally, in Section~\ref{sec:corner inv}, we combine the results from the previous two sections in order to prove our main result, Proposition~\ref{prop:corner inv}.

\subsection{Behaviour of the binomials}\label{sec:binom}

In this section, we provide results regarding the size of the different sets $S_{a,b}$, $S_c^x$, and $S_c^y$ as introduced in Section~\ref{sec:bl}.

\begin{lemma}\label{lem:corner bin}
    Using the notations and assumptions from Section~\ref{sec:bl}, we have that
    \begin{align*}
        \P[\big]{S_{0,b}\neq S_{a,b}}
        \longrightarrow0\,.
    \end{align*}
\end{lemma}

\begin{proof}
    Since $S_{a,b}\subseteq S_{0,b}$, we let $B_1=|S_{0,b}\setminus S_{a,b}|$ so that $S_{0,b}\neq S_{a,b}$ if and only if $B_1\neq0$.
    We also refer to Figure~\ref{fig:corner} for a representation of $B_1$.
    We then observe that $B_1$ is a binomial random variable of size $n$ and with parameter
    \begin{align*}
        p_1
        &=\int_0^{an^{-\gamma}}\int_0^{bn^{-\gamma}}f(x,y)dydx+\int_{an^{-\gamma}}^{bn^{-\gamma}}\int_0^{an^{-\gamma}}f(x,y)dydx\\
        &=\big(1+o(1)\big)\frac{b^{\alpha+1}}{n}\left[\int_0^{a/b}\int_0^1\phi(x,y)dydx+\int_{a/b}^1\int_0^{a/b}\phi(x,y)dydx\right]\,,
    \end{align*}
    where we used~\eqref{eq:f int} for the second equality.
    Further using that $\phi\leq\|\phi\|<\infty$ on $[0,1]^2$, this bound transforms into
    \begin{align*}
        p_1
        \leq\big(1+o(1)\big)\frac{b^{\alpha+1}}{n}\left[\frac{a}{b}+\frac{a}{b}\left(1-\frac{a}{b}\right)\right]\|\phi\|
        \leq\big(1+o(1)\big)\frac{2ab^\alpha\|\phi\|}{n}\,.
    \end{align*}
    The result then follows from Markov's inequality since $\P{B_1\neq0}=\P{B_1\geq1}\leq\E{B_1}=np_1$ and by recalling that $b$ is finite and $a\rightarrow0$.
\end{proof}

\begin{lemma}\label{lem:L bin}
    Using the notations and assumptions from Section~\ref{sec:bl}, for any $t>0$, we have that
    \begin{align*}
        \P*{\Big|S_\epsilon\cap\big(S_b^x\cup S_b^y\big)\Big|\geq tn^\beta}
        \longrightarrow0\,.
    \end{align*}
\end{lemma}

\begin{proof}
    We refer again to Figure~\ref{fig:corner} and let $B_2=|S_\epsilon\cap(S_b^x\cup S_b^y)|$.
    Observe that $B_2$ is again a binomial random variable of size $n$ and with parameter
    \begin{align*}
        p_2
        =\int_0^{bn^{-\gamma}}\int_0^\epsilon f(x,y)dydx+\int_{bn^{-\gamma}}^\epsilon\int_0^{bn^{-\gamma}}f(x,y)dydx\,.
    \end{align*}
    Now, using that $f$ can be approximated by $\phi$ and observing that $x/\epsilon,y/\epsilon\leq1$ in the considered range, we can approximate $p_2$ with
    \begin{align*}
        p_2
        =\big(1+o(1)\big)\epsilon^{\alpha-1}\left[\int_0^{bn^{-\gamma}}\int_0^\epsilon\phi(x/\epsilon,y/\epsilon)dydx+\int_{bn^{-\gamma}}^\epsilon\int_0^{bn^{-\gamma}}\phi(x/\epsilon,y/\epsilon)dydx\right]\,.
    \end{align*}
    Finally, using that $\phi\leq\|\phi\|<\infty$ on $[0,1]^2$, this eventually transforms into
    \begin{align*}
        p_2
        \leq\big(1+o(1)\big)\epsilon^{\alpha-1}\left[\frac{b\epsilon}{n^{1/(\alpha+1)}}+\left(\epsilon-\frac{b}{n^\beta}\right)\frac{b}{n^\beta}\right]\|\phi\|
        \leq\big(2+o(1)\big)\frac{b\epsilon^\alpha}{n^\gamma}\|\phi\|\,.
    \end{align*}
    The desired result then follows from Markov's inequality, since $1-\gamma=\beta$, $b$ is fixed, and $\epsilon\rightarrow0$.
\end{proof}

\begin{lemma}\label{lem:central bin}
    Using the notations and assumptions from Section~\ref{sec:bl}, we have that
    \begin{align*}
        \big|S_{a,b}\big|
        \dconv\poisson[\int_{[0,b]^2}\phi]\,,
    \end{align*}
    where the convergence occurs in distribution.
\end{lemma}

\begin{proof}
    Again, use Figure~\ref{fig:corner} for a visual-aid and let $B_3=|S_{a,b}|$, which is a binomial random variable of size $n$ and with parameter
    \begin{align*}
        p_3
        =\int_{an^{-\gamma}}^{bn^{-\gamma}}\int_{an^{-\gamma}}^{bn^{-\gamma}}f(x,y)dydx
        \sim\frac{b^{\alpha+1}}{n}\int_{a/b}^1\int_{a/b}^1\phi(x,y)dydx\,,
    \end{align*}
    where we used~\eqref{eq:f int} for the last equality.
    We further recall that $\phi(bx,by)=b^{\alpha-1}\phi(x,y)$ and that $a\rightarrow0$, so that $np_3\sim\int_{[0,b]^2}\phi$.
    It is then a known result that binomial random variables with parameters $n$ and $c/n$ converge in distribution to a Poisson random variable with parameter $c$.
\end{proof}

\begin{lemma}\label{lem:strip bin}
    Using the notations and assumptions from Section~\ref{sec:bl}, we have that
    \begin{align*}
        \frac{1}{n^\beta}\left(\frac{\big|S_{c_1}^x\big|}{c_1},\ldots,\frac{\big|S_{c_k}^x\big|}{c_k},\frac{\big|S_{c'_1}^y\big|}{c'_1},\ldots,\frac{\big|S_{c'_\ell}^y\big|}{c'_\ell}\right)
        \dconv(1,\ldots,1)\,,
    \end{align*}
    for any $a\leq c_1,\ldots,c_k\leq b$ and any $a\leq c'_1,\ldots,c'_\ell\leq b$ possibly depending on $n$, where the convergence occurs in distribution (and equivalently in probability).
\end{lemma}

\begin{proof}
    For any $c\in[a,b]$, possibly depending on $n$, using the uniform marginals of a permuton, we observe that $|S_c^x|$ and $|S_c^y|$ are both binomial random variables of size $n$ and with parameter $p_c=cn^{-\gamma}\geq an^{-\gamma}$.
    Recalling that $1-\gamma=\beta>0$ and $a=a_n$ is slowly varying, we see that $np_c\geq an^\beta\rightarrow\infty$, which implies that $|S_c^x|/(np_c)$ and $|S_c^y|/(np_c)$ both converge in probability to $1$.
    By then combining finitely many such events, they all asymptotically hold and the convergence of the lemma follows.
\end{proof}

\subsection{A useful distribution}\label{sec:useful}

After studying basic properties of the sets from Section~\ref{sec:bl}, we now study the distributional limit of a particular combination of random variables, which will prove to be key in the minimal inversion on $[0,\epsilon]^2$.

\begin{proposition}\label{prop:useful-bl}
    Using the notations and assumptions from Section~\ref{sec:bl} and by letting $\Delta$ be the unit square triangle, for any $0<t<b$, we have that
    \begin{align*}
        \P*{\min\Big\{\big|S_{n^\gamma X_i}^x\big|+\big|S_{n^\gamma Y_i}^y\big|:i\in S_{a,b}\Big\}>tn^\beta}
        \longrightarrow\e*{-t^{\alpha+1}\int_\Delta\phi}\,.
    \end{align*}
\end{proposition}

\begin{proof}
    Fix $t$ such that $0<t<b$ and write $Z$ for the random minimum inside the probability.
    Condition the target probability given $S_{a,b}$ and the coordinates of the corresponding points to see that
    \begin{align*}
        \P[\big]{Z>tn^\beta}
        =\E*{\P*{Z>tn^\beta~\Big|~S_{a,b},(X_i,Y_i)_{i\in S_{a,b}}}}
    \end{align*}
    Now thanks to Lemma~\ref{lem:strip bin}, conditionally given the (finite) set $S_{a,b}$ and the positions $(X_i,Y_i)_{i\in S_{a,b}}$, we know that the inner probability satisfies
    \begin{align*}
        \P*{Z>tn^\beta~\Big|~S_{a,b},(X_i,Y_i)_{i\in S_{a,b}}}
        &\sim\P*{\min\Big\{n^\gamma X_i+n^\gamma Y_i:i\in S_{a,b}\Big\}>t~\bigg|~S_{a,b},(X_i,Y_i)_{i\in S_{a,b}}}\\
        &\sim\prod_{i\in S_{a,b}}\I{X_i+Y_i>tn^{-\gamma}}\,.
    \end{align*}
    By dominated convergence, since Lemma~\ref{lem:central bin} tells us that $|S_{a,b}|$ is almost-surely bounded as $n\rightarrow\infty$, and since the distribution of the points in $S_{a,b}$ rely on $f$, we see that, as $n\rightarrow\infty$
    \begin{align*}
        \P[\big]{Z>tn^\beta}
        \sim\E*{\prod_{i\in S_{a,b}}\I{X_i+Y_i>tn^{-\gamma}}}\,.
    \end{align*}
    Now, once we condition on the size $|S_{a,b}|$, the corresponding points in the sets are independently distributed on the square $[an^{-\gamma},bn^{-\gamma}]^2$ according to $f$.
    Recalling that $a\rightarrow0$ and that $b>t$, we see that, for $n$ large enough, in order to have $(x,y)\in(a,b)^2$ such that $x+y<t$, we need to have $a<x<t-a<b$ and then $a<y<t-x<b$.
    It follows that a single point $i\in S_{a,b}$ chosen arbitrarily satisfies
    \begin{align*}
        \P[\Big]{X_i+Y_i>tn^{-\gamma}}
        &=1-\frac{\int_{an^{-\gamma}}^{(t-a)n^{-\gamma}}\int_{an^{-\gamma}}^{tn^{-\gamma}-x}f(x,y)dydx}{\int_{an^{-\gamma}}^{bn^{-\gamma}}\int_{an^{-\gamma}}^{bn^{-\gamma}}f(x,y)dydx}\\
        &=1-\big(1+o(1)\big)\frac{t^{\alpha+1}\int_{a/t}^{(t-a)/t}\int_{a/t}^{1-x}\phi(x,y)dydx}{b^{\alpha+1}\int_{a/b}^1\int_{a/b}^1\phi(x,y)dydx}\,.
    \end{align*}
    Recalling that $\phi(bx,by)=b^{\alpha-1}\phi(x,y)$ and that $a\rightarrow0$, we see that
    \begin{align*}
        \P[\Big]{X_i+Y_i>tn^{1/(\alpha+1)}}
        =1-\big(1+o(1)\big)t^{\alpha+1}\frac{\int_\Delta\phi}{\int_{[0,b]^2}\phi}\,.
    \end{align*}
    Combining the previous results and using again the dominated convergence theorem, we see that
    \begin{align*}
        \P[\big]{Z>tn^\beta}
        \sim\E*{\left(1-\big(1+o(1)\big)t^{\alpha+1}\frac{\int_\Delta\phi}{\int_{[0,b]^2}\phi}\right)^{|S_{a,b}|}}\,.
    \end{align*}
    To conclude the proof, observe that $b$ is fixed so that the term inside the innermost brackets converges to a constant.
    Combining this with Lemma~\ref{lem:central bin}, which tells us that $|S_{a,b}|$ converges to a Poisson random variable with parameter $\int_{[0,b]^2}\phi$, the previous asymptotic behaviour becomes
    \begin{align*}
        \P[\big]{Z>tn^\beta}
        \sim\e*{-\big(1+o(1)\big)t^{\alpha+1}\int_\Delta\phi}\,.
    \end{align*}
    Finally, using that the terms inside the exponential are finite, the desired convergence follows.
\end{proof}

\subsection{Minimal inversion in the corner}\label{sec:corner inv}

We conclude Section~\ref{sec:bl} by proving the following proposition, which proves the convergence in distribution of the minimal inversion in the corner $[0,\epsilon]^2$.

\begin{proposition}\label{prop:corner inv}
    Using the notations from Section~\ref{sec:notation} and Section~\ref{sec:bl}, we have that
    \begin{align*}
        n^{-\beta}\minv_n(S_\epsilon)
        \dconv\ray[\alpha+1,\int_\Delta\phi]\,,
    \end{align*}
    where $\Delta$ is the unit square triangle and the convergence occurs in distribution.
\end{proposition}

Before proving the proposition, we prove two intermediary lemmas, controlling the minimal inversion in $S_{a,b}$ and in $S_\epsilon\setminus S_{a,b}$.

\begin{lemma}\label{lem:corner inv}
    Using the notations from Section~\ref{sec:bl}, we have that
    \begin{align*}
        n^{-\beta}\left[\minv_n(S_{a,b})-\min\Big\{\big|S_{n^\gamma X_i}^x\big|+\big|S_{n^\gamma Y_i}^y\big|:i\in S_{a,b}\Big\}\right]
        \dconv0\,,
    \end{align*}
    where the convergence occurs in distribution (and equivalently in probability).
    In fact, we can replace $n^{-\beta}$ by any sequence converging to $0$ with $n$.
\end{lemma}

\begin{proof}
    Consider an arbitrary $i\in S_{a,b}$.
    Referring to Figure~\ref{fig:corner} and using that no two points have the same coordinate, we see that
    \begin{align*}
        \big|S_{n^\gamma X_i}^x\big|+\big|S_{n^\gamma Y_i}^y\big|-\big|S_{0,b}\big|\leq
        \big|\inv(i)\big|
        \leq\big|S_{n^\gamma X_i}^x\big|+\big|S_{n^\gamma Y_i}^y\big|\,.
    \end{align*}
    Since $i$ is arbitrary, it follows that
    \begin{align*}
        0\leq
        \min\Big\{\big|S_{n^\gamma X_i}^x\big|+\big|S_{n^\gamma Y_i}^y\big|:i\in S_{a,b}\Big\}-\minv_n(S_{a,b})
        \leq\big|S_{0,b}\big|\,.
    \end{align*}
    The result of the lemma then follows from Lemma~\ref{lem:corner bin} telling us that $S_{0,b}=S_{a,b}$ with high probability, and Lemma~\ref{lem:central bin} telling us that $S_{a,b}$ converges in distribution to a finite Poisson random variable. 
\end{proof}

\begin{lemma}\label{lem:almost corner}
    Using the notations from Section~\ref{sec:bl}, for any $t<b$, we have that
    \begin{align*}
        \P[\Big]{\minv_n(S_\epsilon\setminus S_{a,b})>tn^\beta}
        \longrightarrow1\,.
    \end{align*}
\end{lemma}

\begin{proof}
    Relying on Figure~\ref{fig:corner}, any point within $[bn^{-\gamma},\epsilon]\times[0,\epsilon]$ has at least as many inversion as there are points within $[0,bn^{-\gamma}]\times[\epsilon,1]$, which means that
    \begin{align*}
        \minv_n\big([bn^{-\gamma},\epsilon]\times[0,\epsilon]\big)
        \geq\big|S_b^x\big|-\big|S_\epsilon\cap S_b^x\big|
        \geq\big|S_b^x\big|-\Big|S_\epsilon\cap\big(S_b^x\cup S_b^y\big)\Big|\,,
    \end{align*}
    where the second lower bound simply uses that $S_b^x\subseteq S_b^x\cup S_b^y$.
    A similar argument where we swap the roles of the $x$ and $y$ coordinates leads to
    \begin{align*}
        \minv_n\big([0,\epsilon]\times[bn^{-\gamma},\epsilon]\big)
        \geq\big|S_b^y\big|-\Big|S_\epsilon\cap\big(S_b^x\cup S_b^y\big)\Big|\,,
    \end{align*}
    Finally, using $S_{a,b}=S_{0,b}$ with high probability, thanks to Lemma~\ref{lem:corner bin}, it follows that
    \begin{align*}
        \P[\Big]{\minv_n(S_\epsilon\setminus S_{a,b})>tn^\gamma}
        \geq\P[\Big]{\min\big\{\big|S_b^x\big|,\big|S_b^y\big|\big\}-\Big|S_\epsilon\cap\big(S_b^x\cup S_b^y\big)\Big|>tn^\gamma}+o(1)\,.
    \end{align*}
    Using the convergence result from Lemma~\ref{lem:strip bin}, this lower bound transforms into
    \begin{align*}
        \P[\Big]{\minv_n(S_\epsilon\setminus S_{a,b})>tn^\gamma}
        \geq\P[\Big]{-\Big|S_\epsilon\cap\big(S_b^x\cup S_b^y\big)\Big|>(t-b)n^\gamma}+o(1)\,,
    \end{align*}
    which then, combined with Lemma~\ref{lem:L bin}, tells us that the right-hand side converges to $1$ when $b-t>0$, exactly as desired.
\end{proof}

We now have all the tools to prove our main proposition.

\begin{proof}[Proof of Proposition~\ref{prop:corner inv}]
    We rely on the definitions and properties from Section~\ref{sec:bl}.
    We further fix some $t>0$ and assume that $b>t$.
    From their respective definitions, observe that
    \begin{align*}
        \minv_n(S_\epsilon)
        =\min\big\{\minv_n(S_{a,b}),\minv_n(S_\epsilon\setminus S_{a,b})\big\}\,.
    \end{align*}
    First, thanks to Lemma~\ref{lem:almost corner}, we see that
    \begin{align*}
        \P[\Big]{\minv_n(S_\epsilon)>tn^\beta}
        &=\P[\Big]{\minv_n(S_{a,b})>tn^\beta,\minv_n(S_\epsilon\setminus S_{a,b})>tn^\beta}\\
        &=\P[\Big]{\minv_n(S_{a,b})>tn^\beta}+o(1)\,.
    \end{align*}
    Then, using Lemma~\ref{lem:corner inv}, the previous probability transforms into
    \begin{align*}
        \P[\Big]{\minv_n(S_\epsilon)>tn^\beta}
        =\P*{\min\Big\{\big|S_{n^\gamma X_i}^x\big|+\big|S_{n^\gamma Y_i}^y\big|:i\in S_{a,b}\Big\}>tn^\beta}+o(1)\,.
    \end{align*}
    To conclude the proof, use Proposition~\ref{prop:useful-bl} to obtain that
    \begin{align*}
        \P[\Big]{\minv_n(S_\epsilon)>tn^\beta}
        =\e*{-t^{\alpha+1}\int_\Delta\phi}+o(1)\,,
    \end{align*}
    which is exactly the desired generalized Rayleigh distribution.
\end{proof}

\section{Standard permutons}

The goal of this section is to prove the statement of Theorem~\ref{thm:main}.
We start by extracting some properties of corner-balanced permutons in Section~\ref{sec:corner}, before combining these results along with those of Section~\ref{sec:main} to prove Theorem~\ref{thm:main} in Section~\ref{sec:main}.

\subsection{Corner-balanced permutons}\label{sec:corner}

In this section, we show that the nodes outside of the corners in a corner-balanced permuton have linear inversion number.
We start with a lemma which does not rely on smoothness before extending the results to standard permutons.

\begin{lemma}\label{lem:top-bottom}
    Let $\mu$ be a corner-balanced permuton.
    Using the notations from Section~\ref{sec:notation}, for any $0<\epsilon\leq1/2$, there exists $0<\rho<1$ and $\delta>0$ such that
    \begin{align*}
        \P*{\frac{1}{n}\minv_n\Big(\big([0,\epsilon]\times[0,\rho]\cup[1-\epsilon,1]\times[\rho,1]\big)^c\Big)\geq\delta}
        \longrightarrow1\,.
    \end{align*}
\end{lemma}

\begin{proof}
    First, fix $0<\epsilon\leq1/2$ and find $0<\rho<1$ according to the corner-balanced property of the permuton.
    Let $B_-=|\{i:X_i\leq\epsilon,Y_i\geq\rho\}|$ be the number of top-left points and $B_+=|\{i:X_i\geq1-\epsilon,Y_i\leq\rho\}|$ be the number of bottom-right point.
    Since $\mu$ is corner-balanced, $B_\pm$ are binomial random variables of size $n$ and strictly positive parameter $p_\pm$.
    Note that they are not independent, nor are they necessarily identically distributed, but this does not affect their asymptotic behaviour.
    Thus, we see that $B=\min\{B_-,B_+\}$ divided by $n$ converges almost-surely as $n\rightarrow\infty$ towards $\min\{p_-,p_+\}>0$.
    To conclude the proof, simply note that $B$ is a lower bound for the minimal inversion on the desired set, so any $0<\delta<\min\{p_-,p_+\}$ satisfies the desired convergence.
\end{proof}

\begin{lemma}\label{lem:left}
    Let $\mu$ be a corner-balanced and BL-smooth permuton.
    Using the notations from Section~\ref{sec:notation}, for any $0<\epsilon\leq1/2$ small enough, there exists $\delta>0$ such that
    \begin{align*}
        \P*{\frac{1}{n}\minv_n\big([0,\epsilon]\times(\epsilon,1]\big)\geq\delta}
        \longrightarrow1\,.
    \end{align*}
\end{lemma}

\begin{proof}
    Using the BL-smoothness of $\mu$, and since $\phi>0$, we can find $\epsilon_0$ such that $\mu$ has a strictly positive density on $[0,\epsilon]^2$.
    Choose now $\epsilon\leq\epsilon_0$ and use the corner-balanced property of $\mu$ to find $0<\rho<1$ such that $\mu([0,\epsilon/2]\times[\rho,1])>0$.
    Using the same argument as for Lemma~\ref{lem:top-bottom}, we know that
    \begin{align*}
        \P*{\frac{1}{n}\minv_n\big([0,\epsilon]\times[\rho,1]\big)\geq\delta}
        \longrightarrow1\,.
    \end{align*}
    Thus, we now focus on points such that $\epsilon<Y_i<\rho$.

    Using that $\mu$ has a strictly positive density on $[0,\epsilon]^2$, we have that $\mu((\epsilon/2]\times[0,\epsilon])>0$.
    Thus, by using a similar technique as that of Lemma~\ref{lem:top-bottom} again, and by letting $\delta>0$ be such that
    \begin{align*}
        \delta
        <\min\Big\{\mu\big([0,\epsilon/2]\times[\rho,1]\big),\mu\big((\epsilon/2]\times[0,\epsilon]\big)\Big\}\,,
    \end{align*}
    the minimal inversion of any point such that $0\leq X_i\leq\epsilon$ and $\epsilon<Y_i<\rho$ is asymptotically larger than $n\delta$, thus leading to the desired result.
\end{proof}

\subsection{Minimal inversion of a standard permuton}\label{sec:main}

In this section, we combine Proposition~\ref{prop:corner inv} from Section~\ref{sec:corner inv} along with Lemma~\ref{lem:top-bottom} and Lemma~\ref{lem:left} to prove Theorem~\ref{thm:main}.

\begin{proof}[Proof of Theorem~\ref{thm:main}]
    We recall that $\mu$ here is a $(\alpha,\phi)$-standard permuton.
    We further denote by $(\alpha_-,\phi_-)$ and $(\alpha_+,\phi_+)$ respectively the BL and TR smoothness parameters, and invite the reader to look at Section~\ref{sec:ass} for some properties and relationships between these variables.
    
    We start by observing that the symmetry around $y=1-x$ between BL-smooth and TR-smooth, along with Lemma~\ref{lem:left} tells us that, for $\epsilon$ small enough, we have that
    \begin{align*}
        \min\left\{\P*{\frac{1}{n}\minv_n\big([0,\epsilon]\times(\epsilon,1]\big)\geq\delta},
        \P*{\frac{1}{n}\minv_n\big([1-\epsilon,1]\times[0,1-\epsilon)\big)\geq\delta}\right\}
        \longrightarrow1\,.
    \end{align*}
    Combining this with Lemma~\ref{lem:top-bottom}, along with the fact that $\minv_n(A)$ is the minimal value of any partition $A_1,\ldots,A_k$ of $A$, we obtain that
    \begin{align*}
        \P*{\frac{1}{n}\minv_n\Big(\big(S_\epsilon\cup S_\epsilon^{-1}\big)^c\Big)\geq\delta}
        \longrightarrow1\,.
    \end{align*}
    Since any $\epsilon>0$ admits such a $\delta>0$, for any $\eta<1$ such that $\eta>\beta$, we can find $\epsilon=\epsilon_n$ slowly varying and converging to $0$ such that
    \begin{align*}
        \P*{\minv_n\Big(\big(S_\epsilon\cup S_\epsilon^{-1}\big)^c\Big)\geq n^\eta}
        \longrightarrow1\,.
    \end{align*}
    We now fix such $\eta$ and $\epsilon=\epsilon_n$.
    Using again that $\minv_n=\minv_n([0,1]^2)$ can be decomposed from a partition of $[0,1]^2$, the previous convergence tells us that
    \begin{align}\label{eq:BLTR corner}
        \P[\Big]{\minv_n>tn^\beta}
        =\P[\Big]{\minv_n(S_\epsilon)>tn^\beta,\minv_n(S_\epsilon^{-1})>tn^\beta}+o(1)\,.
    \end{align}
    It now suffices to study this probability.
    
    Using Proposition~\ref{prop:corner inv} and the symmetry between BL and TR-smooth along with the current notations, we know that, for any $t>0$, we have
    \begin{align*}
        \P[\Big]{\minv_n(S_\epsilon^\pm)>tn^{\beta_\pm}}
        \longrightarrow\e*{-t^{\alpha_\pm+1}\int_\Delta\phi_\pm}\,,
    \end{align*}
    where $S_\epsilon^-=S_\epsilon$ and $S_\epsilon^+=S_\epsilon^{-1}$.
    We now observe that $\minv_n(S_\epsilon)$ and $\minv_n(S_\epsilon^{-1})$ are asymptotically independent.
    Indeed, retracing the proofs from Section~\ref{sec:bl} for both the bottom-left and top-right corners, Lemma~\ref{lem:corner bin} and Lemma~\ref{lem:L bin} both rely on standard bounds for binomial random variables, so their symmetric counterpart also satisfy this convergence.
    Then the Poisson limit of $|S_{a,b}|$ from Lemma~\ref{lem:central bin} holds jointly with that of $|S_{a,b}^{-1}|$, with the appropriate parameters.
    Furthermore, the result from Lemma~\ref{lem:strip bin}, which relies on convergence of binomial random variables, also applies when considering the top and right intervals, instead of the bottom and left ones (respectively $S_c^y$ and $S_c^x$).
    Finally, Proposition~\ref{prop:useful-bl} simply combines the previous results along with the distribution close to $(0,0)$, so it naturally extends independently to the points close to $(1,1)$.
    To conclude, Proposition~\ref{prop:corner inv}, which simply relies on the aforementioned results, can then be extended to both bottom-left and top-right corners, telling us that the pair $n^{-\beta_-}\minv_n(S_\epsilon),n^{-\beta_+}\minv_n(S_\epsilon^{-1})$ jointly converges to a pair of independent generalized Rayleigh random variables with the appropriate parameters.

    Rewrite the right-hand side probability from~\eqref{eq:BLTR corner} to make powers of $\alpha_\pm/(\alpha_\pm+1)$ appear:
    \begin{align*}
        \P[\Big]{\minv_n(S_\epsilon)>tn^\beta,\minv_n(S_\epsilon^{-1})>tn^\beta}
        =\P[\bigg]{\minv_n(S_\epsilon)>\Big(tn^{\beta-\beta_-}\Big)n^{\beta_-},\minv_n(S_\epsilon^{-1})>\Big(tn^{\beta-\beta_+}\Big)n^{\beta_+}}\,.
    \end{align*}
    Since $\beta-\beta_\pm\leq0$ and equal to $0$ if and only if $\alpha_\pm=\alpha=\min\{\alpha_0,\alpha_+\}$, or equivalently, if and only if $\alpha_\pm\leq\alpha_\mp$, the previous result combined with the asymptotic independence and the application of Proposition~\ref{prop:corner inv} to both bottom-left and top-right corners tells us that
    \begin{align*}
        \P[\Big]{\minv_n(S_\epsilon)>tn^\beta,\minv_n(S_\epsilon^{-1})>tn^\beta}
        &\sim\e*{-\Big(tn^{\beta-\beta_-}\Big)^{\alpha_-+1}\int_\Delta\phi_--\Big(tn^{\beta-\beta_+}\Big)^{\alpha_++1}\int_\Delta\phi_+}\\
        &\sim\e*{-t^{\alpha+1}\I{\alpha_-\leq\alpha_+}\int_\Delta\phi_--t^{\alpha+1}\I{\alpha_+\leq\alpha_-}\int_\Delta\phi_+}\,.
    \end{align*}
    Finally, using that $\phi=\phi_-\I{\alpha_-\leq\alpha_+}+\phi_+\I{\alpha_+\leq\alpha_-}$, this convergence along with~\eqref{eq:BLTR corner} exactly corresponds to the statement of Theorem~\ref{thm:main}.
\end{proof}

\section{Tile-standard permutons}\label{sec:tile-standard}

From now on, we focus on the case of tile-standard permutons.
We invite the reader to take a look at Section~\ref{sec:tile} again, since it introduces most of the relevant terminology.
For the rest of this section $\mu$ is a $(\rect,\alpha,\phi)$-standard permuton, with $\rect=\{S_1,\ldots,S_k\}$ a tiling.
For each $\ell\in[k]$, $(x_\ell,y_\ell)$ denotes the bottom-left corner of $S_\ell$ and $\delta_\ell>0$ its side length (since it is a square), meaning that $S_\ell=[x_\ell,x_\ell+\delta_\ell]\times[y_\ell,y_\ell+\delta_\ell]$.
We also have that $\mu_\ell=\mu_{/S_\ell}$ is $(\alpha_\ell,\phi_\ell)$-standard, meaning that the corresponding linearly rescaled permuton on $[0,1]^2$ is $(\alpha_\ell,\delta_\ell^{\alpha_\ell}\phi_\ell)$-standard.
We finally recall that $\sigma_\rect$ is the permutation corresponding to $\rect$, that $\diag_\rect$ is the set of splittings, and that $\mass_\rect$, defined in~\eqref{eq:mass}, is the minimal anti-diagonal mass, relying on the function
\begin{align*}
    L_\rect(\ell)
    =\sum_{m\in\inv_{\sigma_\rect}(\ell)}\sqrt{\lambda_2(S_m)}
    =\sum_{m\in\inv_{\sigma_\rect}(\ell)}\delta_m\,.
\end{align*}
We further drop the dependency on $\sigma_\rect$ from $\inv$.

In order to study the minimal inversion of a permuton organized by a tiling, we individually study the minimal inversion on each square.
To do so, we denote by
\begin{align*}
    N_\ell
    =\big|\big\{i:(X_i,Y_i)\in S_\ell\big\}\big|
\end{align*}
the number of points within $S_\ell$, having in mind that each point belongs to a unique square with probability $1$ so that $N_1+\ldots+N_k=n$.
We further rely on the notations from Section~\ref{sec:notation} and let $\minv_n(S_\ell)$ be the minimal inversion of the points falling within $S_\ell$ as well as use $\beta_1,\ldots,\beta_k$ and $\gamma_1,\ldots,\gamma_k$ for the parameters corresponding to $\alpha_1,\ldots,\alpha_k$.

The rest of this section is organized as follows.
In Section~\ref{sec:coherent} we prove Proposition~\ref{prop:unique tile} which states that the parameters of a tile-standard permuton are unique.
In Section~\ref{sec:non-diag}, we study the minimal inversion of points within each square.
Finally, in Section~\ref{sec:tile result}, we prove Theorem~\ref{thm:tile}, which concludes the set of results for this article.

\subsection{Coherence of the tile-standard definition}\label{sec:coherent}

We start this section by showing that any permuton which can be decomposed according to some tiling with more than $2$ sets cannot be standard.
This comes from the fact that smoothness requires non-zero density around $(0,0)$ and $(1,1)$, whereas the corner-balanced property forces the distribution to have some non-zero mass around $(0,1)$ and $(1,0)$.
A tiling $\rect=\{S_1,\ldots,S_k\}$ is said to be \textit{adapted} to $\mu$ if
\begin{align*}
    \mu\Big(S_1\cup\cdots\cup S_k\Big)
    =1\,.
\end{align*}

\begin{lemma}\label{lem:tile ass}
    Let $\mu$ be a permuton and $\rect=\{S_1,\ldots,S_k\}$ be a tiling adapted to $\mu$.
    Denote by $\sigma=\sigma_\rect$ the permutation corresponding to $\rect$.
    Then we have the following properties.
    \begin{itemize}
        \item If $\sigma(1)<\sigma(k)$, then $\mu$ is not corner-balanced.
        \item If $\sigma(1)\neq1$, then $\mu$ is not BL-smooth.
        \item If $\sigma(k)\neq k$, then $\mu$ is not TR-smooth.
    \end{itemize}
    It follows that, if $k\geq2$, then $\mu$ is not standard.
\end{lemma}

\begin{proof}
    We assume without loss of generality that $S_1\prec_x\cdots\prec_xS_k$ are ordered according to their horizontal intervals, meaning that $S_1$ is the leftmost rectangle and $S_k$ is the rightmost one.
    Further find $\epsilon>0$ small enough such that it is smaller than the side length of both $S_1$ and $S_k$.
    We now consider the different cases.

    If $\sigma(1)<\sigma(k)$, then there exists $r$ such that $\max\{y:(x,y)\in S_1\}\leq r\leq\min\{y:(x,y)\in S_k\}$.
    Consider now any $0<\rho<1$.
    If $\rho\geq r$, then $[0,\epsilon]\times[\rho,1]$ intersects the union of the rectangles at most on the single line ($[0,\epsilon]\times\{\rho\}$), so it has measure  $0$ and $\mu([0,\epsilon]\times[\rho,1])=0$.
    Conversely, if $\rho\leq r$, then $\mu([1-\epsilon,1]\times[0,\rho])=0$.
    In both cases, the minimum of the two measures is $0$ and $\mu$ is not corner-balanced.

    Since the other two properties are symmetric with respect to the smoothness properties, we only prove the first one.
    When $\sigma(1)\neq1$, then $\min\{y:(x,y)\in S_1\}>0$ and so $\mu([0,\epsilon]^2)=0$.
    This is not compatible with the BL-smoothness as $\phi$ is assumed to be strictly positive almost-everywhere.

    To conclude the proof, simply use the contraposite of the three statements to observe that if $\mu$ is standard, then $\sigma(1)\geq\sigma(k)$, $\sigma(1)=1$, and $\sigma(k)=k$, so that necessarily $k=1$.
\end{proof}

With this lemma, we now prove the main result of this section, stating that the parameters of a tile-standard permuton are uniquely defined.

\begin{proposition}\label{prop:unique tile}
    Let $\mu$ be a tile-standard permuton as defined in Section~\ref{sec:tile}.
    Then there exists a unique triplet $(\rect,\alpha,\phi)$ such that $\mu$ is $(\rect,\alpha,\phi)$-tile-standard.
\end{proposition}

\begin{proof}
    Observe that if we show that the tiling is unique, then the uniqueness of $\alpha$ and $\phi$ follows from the definition.
    Consider now two tilings $\rect=\{S_1,\ldots,S_k\}$ and $\rect'=\{S'_1,\ldots,S'_\ell\}$ adapted to $\mu$.
    Then the new tiling $\rect''=\{R''_1,\ldots,R''_m\}$ obtained by taking the intersection of the squares from $\rect$ and $\rect'$ and only keeping the ones with non-empty interior is itself a tiling adapted to $\mu$, since
    \begin{align*}
        \mu\Big(S''_1\cup\cdots\cup S''_m\Big)
        =\mu\bigg(\Big(S_1\cup\cdots\cup S_k\Big)\cap\Big(S'_1\cup\cdots\cup S'_\ell\Big)\bigg)
        =1
    \end{align*}
    and since the uniform margins of $\mu$ prevents any intersection to not be a square.
    This implies that there exists a minimal tiling adapted to $\mu$ (not necessarily requiring the corresponding decomposition to be into standard sub-permutons).
    But then, for any adapted tiling $\rect=\{S_1,\ldots,S_k\}$ which is not minimal, there exists a sub-permuton $\mu_{/S_\ell}$ which will be decomposed into more than $2$ new sub-permutons when applying the minimal tiling.
    Applying Lemma~\ref{lem:tile ass}, $\mu_{/S_\ell}$ is not standard and thus $\mu$ cannot be $(\rect,\alpha.\phi)$-tile-standard for any $\alpha$ and $\phi$.
    This proves the uniqueness of $\rect$ and thus also that of $\alpha$ and $\mu$.
\end{proof}

\subsection{Minimal inversion on the squares}\label{sec:non-diag}

In this section, we study the convergence in distribution of all the individual minimal inversions on the squares $S_1,\ldots,S_k$.
We start by showing in Proposition~\ref{prop:diag} the Rayleigh-distributed limit of splittings, before showing in Proposition~\ref{prop:non-diag} that non-splittings have minimal inversion of order $n$.

\begin{proposition}\label{prop:diag}
    Using the notations and assumptions from Section~\ref{sec:tile-standard}, and by letting $\Delta$ be the unit square triangle, for any $\ell\in\diag_\rect$, $t>0$, and $n_\ell=n_{\ell,n}\rightarrow\infty$, we have
    \begin{align*}
        \P*{\minv_n(S_\ell)>tn_\ell^{\beta_\ell}~\Big|~N_\ell=n_\ell,\big\{(X_i,Y_i)\big\}\cap S_\ell^c}
        \longrightarrow\e*{-\delta_\ell^{\alpha_\ell}t^{\alpha_\ell+1}\int_\Delta\phi_\ell}\,.
    \end{align*}
\end{proposition}

\begin{proof}
    We start by observing that no inversion is created between points within $S_\ell$ and points within other squares $S_m$, with $m\neq\ell$.
    This implies that the variable $\minv(S_\ell)$ is independent of $\{(X_i,Y_i)\}\cap S_\ell^c$.
    Since inversions are created by the relative order of points, by letting $\Tilde{\mu_\ell}$ be the transformation of $\mu_\ell$ into a permuton on $[0,1]^2$, it follows that
    \begin{align*}
        \P*{\minv_n(S_\ell)>tn_\ell^{\beta_\ell}~\Big|~N_\ell=n_\ell,\big\{(X_i,Y_i)\big\}\cap S_\ell^c}
        =\P[\Big]{\minv_{n_\ell}^{\Tilde{\mu}_\ell}>tn_\ell^{\beta_\ell}}\,.
    \end{align*}
    Now, recall from Section~\ref{sec:tile} that $\mu_\ell$ is $(\alpha_\ell,\phi_\ell)$-standard if and only if $\Tilde{\mu}_\ell$ is $(\alpha_\ell,\delta_\ell^{\alpha_\ell}\phi_\ell)$-standard.
    By applying Theorem~\ref{thm:main} to the right-hand side, we obtain that
    \begin{align*}
        \P*{\minv_n(S_\ell)>tn_\ell^{\beta_\ell}~\Big|~N_\ell=n_\ell,\big\{(X_i,Y_i)\big\}\cap S_\ell^c}
        \longrightarrow\e*{-t^{\alpha_\ell+1}\int_\Delta \delta_\ell^{\alpha_\ell}\phi_\ell}\,,
    \end{align*}
    exactly as desired.
\end{proof}

\begin{proposition}\label{prop:non-diag}
    Using the notations and assumptions from Section~\ref{sec:tile-standard}, for any $\ell\notin\diag_\rect$, we have
    \begin{align*}
        \frac{\minv_n(S_\ell)}{n}
        \dconv L_\rect(\ell)\,,
    \end{align*}
    where the convergence occurs in distribution (and equivalently in probability).
\end{proposition}

\begin{proof}
    Observe first that, since $k$ is finite, the random vector $(N_1,\ldots,N_k)$ normalized by $n$ converges (almost-surely) to $(\delta_1,\ldots,\delta_k)$, where $\delta_\ell=\mu(S_\ell)$ is the side length of $S_\ell$ and the equality holds thanks to the uniform marginals of $\mu$.
    Moreover, for any $m\in\inv(\ell)$, the points in $S_m$ form an inversion with the points in $S_\ell$, and for any $m\notin\{\ell\}\cup\inv(\ell)$, the points in $S_m$ do not create a inversions with points from $S_\ell$.
    Finally, conditionally given $N_\ell=n_\ell$, the distribution of the minimal inversion created by points within $S_\ell$ is distributed as $\minv_{n_\ell}^{\Tilde{\mu}_\ell}$, where $\Tilde{\mu}_\ell$ is the linear transformation of $\mu_\ell$ into a permuton on $[0,1]^2$.
    It follows that
    \begin{align*}
        \frac{\minv_n(S_\ell)}{n}
        \overset{d}{=}\frac{\Tilde{\minv}_{N_\ell}^{\Tilde{\mu}_\ell}}{n}+\frac{1}{n}\sum_{m\in\inv(\ell)}N_m\,.
    \end{align*}
    where the equality is in distribution and, conditionally given $N_1,\ldots,N_k$, $\Tilde{\minv}_{N_\ell}^{\Tilde{\mu}_\ell}$ is the random minimal inversion of $N_\ell$ points distributed according to $\Tilde{\mu}_\ell$.
    To conclude, recall that $\Tilde{\mu}_\ell$ is assumed to be standard, so this random variable is of order $n^{\beta_\ell}=o(n)$.
    The desired convergence then follows from the convergence of $(N_1,\ldots,N_k)$ and the definition of $L_\rect(\ell)$ (in Section~\ref{sec:tile-standard}).
\end{proof}

\subsection{Minimal inversion of tile-standard permutons}\label{sec:tile result}

In this section, we combine the previous results to prove the limiting distribution of the minimal inversion of a tile-standard permuton.

\begin{proof}[Proof of Theorem~\ref{thm:tile}]
    Instead of directly proving the statement as found in the theorem, we prove the two different convergence results explained right below it, corresponding to whether $\diag_\rect$ is empty or not.

    When $\diag_\rect$ is empty, we know that Proposition~\ref{prop:non-diag} applies to all $\ell\in[k]$.
    Further using that $\minv_n$ is the minimum of all $\minv_n(S_\ell)$ for $\ell\in[k]$ along with the fact that $k$ is finite, so that the individual convergence results imply the joint convergence, we directly see that
    \begin{align*}
        \frac{\minv_n}{n}
        =\min\left\{\frac{\minv_n(S_\ell)}{n}:\ell\in[k]\right\}
        \dconv\min\Big\{L_\rect(\ell):\ell\in[k]\Big\}
        =\mass_\rect\,.
    \end{align*}
    We now assume that $\diag_\rect\neq\emptyset$.

    First of all, since $\alpha$ is finite in this case (so that $\beta<1$) and Proposition~\ref{prop:non-diag} still applies to all $\ell\notin\diag_\rect$, we directly observe that
    \begin{align*}
        \P[\Big]{\minv_n>tn^\beta}
        =\P[\Big]{\minv_n(S_\ell)>tn^\beta,\forall\ell\in[k]}
        =\P[\Big]{\minv_n(S_\ell)>tn^\beta,\forall\ell\in\diag_\rect}+o(1)\,.
    \end{align*}
    Now, using that $n^{-1}(N_1,\ldots,N_k)$ converges almost-surely to $(\delta_1,\ldots,\delta_\ell)$, we can further condition on $N_\ell$ with $\ell\in\diag_\rect$ and obtain that
    \begin{align*}
        \P[\Big]{\minv_n>tn^\beta}
        &=\P*{\minv_n(S_\ell)>tn^\beta,\forall\ell\in\diag_\rect~\bigg|~N_\ell=\big(\delta_\ell+o(1)\big)n,\forall\ell\in\diag_\rect}+o(1)\,.
    \end{align*}
    Using Proposition~\ref{prop:diag}, which tells us that the minimal inversion on the different splittings are independent conditionally given their sizes, we see that
    \begin{align*}
        \P[\Big]{\minv_n>tn^\beta}
        &=\prod_{\ell\in\diag_\rect}\P*{\minv_n(S_\ell)>tn^\beta~\bigg|~N_\ell=\big(\delta_\ell+o(1)\big)n}+o(1)\\
        &=\prod_{\ell\in\diag_\rect}\P*{\minv_n(S_\ell)>\big(t\delta_\ell^{-\beta}+o(1)\big)N_\ell^\beta~\bigg|~N_\ell=\big(\delta_\ell+o(1)\big)n}+o(1)\,.
    \end{align*}
    It now suffices to apply Proposition~\ref{prop:diag} again, along with the fact that $\alpha_\ell\geq\alpha$, to transform this probability into
    \begin{align*}
        \P[\Big]{\minv_n>tn^\beta}
        &=\prod_{\ell\in\diag_\rect}\e*{-\delta_\ell^\alpha\big(t\delta_\ell^{-\beta}+o(1)\big)^{\alpha+1}\I{\alpha_\ell=\alpha}\int_\Delta\phi_\ell}+o(1)\\
        &=\e*{-t^{\alpha+1}\int_\Delta\left(\sum_{\ell\in\diag_\rect}\phi_\ell\I{\alpha_\ell=\alpha}\right)}+o(1)\,.
    \end{align*}
    Using the definition of $\phi$ from the tile-standard definition in Section~\ref{sec:tile} and recalling that $\mass_\rect=0$ when $\diag_\rect\neq\emptyset$, this is exactly the desired convergence as stated in the theorem.
\end{proof}

\section{Useful examples}

In this section, we provide specific examples of permutons to highlight the interest and limits of our results.
In Section~\ref{sec:standard example}, we show that any $\alpha>0$ admits a $(\alpha,\phi)$-standard permuton (for an appropriate $\phi$).
In Section~\ref{sec:counter}, we provide a few non-standard permutons, highlighting the roles of the assumptions and the difficulty in characterizing universal minimal inversion behaviour.

\subsection{A family of standard permutons}\label{sec:standard example}

For any $\alpha>0$, consider the function $h:[0,1]^2\rightarrow[0,\infty)$ defined by
\begin{align}\label{eq:gamma}
    h(x,y)
    =\left\{\begin{array}{ll}
        (x+y)^{\alpha-1} & \textrm{if $x+y\leq1$} \\
        (2-x-y)^{\alpha-1} & \textrm{if $x+y>1$}\,.
    \end{array}\right.
\end{align}
When $\alpha<1$, we simply let $h(0,0)=h(1,1)=0$, which does not affect the value of $h$ upon integration.
We observe that $h$ is symmetric around $y=1-x$ and we provide a representation of $h$ for different values of $\alpha$ in Figure~\ref{fig:gamma}.
We now show that $h$ can be transformed into the density of a $(\alpha,2\phi)$-standard permuton, with $\phi(x,y)=\alpha^{\alpha+1}(x+y)^{\alpha-1}$.

\begin{figure}[htb]
    \centering
    \includegraphics[width=\textwidth]{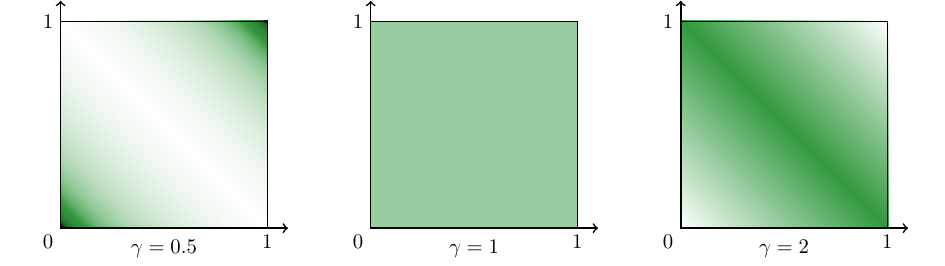}
    \caption{
    A heatmap of the function $h$ from~\eqref{eq:gamma} for different values of $\alpha$.
    When $\alpha<1$, the function diverges around $(0,0)$ and $(1,1)$.
    When $\alpha=1$, the function is constant equal to $1$ on $[0,1]^2$.
    When $\alpha>1$, the function is minimal at $(0,0)$ and $(1,1)$, with value $0$, and maximal on $\{(x,1-x):x\in[0,1]\}$, with value $1$.
    Observe that $h$ is a proper permuton only when $\alpha=1$, however it can be transformed into a permuton, using~\eqref{eq:gamma marginal}, and this permuton is of the same order as $h$ around $(0,0)$ and $(1,1)$, as shown in~\eqref{eq:gamma hf}.
    }
    \label{fig:gamma}
\end{figure}

Denote by $\eta(t)=\int_0^t\int_0^1h$ the marginal of $h$, both for $x$ and $y$ thanks to the symmetry of $h$.
Then $\eta$ is a strictly monotone function and the function
\begin{align}\label{eq:gamma marginal}
    f(x,y)
    =\frac{h\big(\eta^{-1}(x),\eta^{-1}(y)\big)}{\eta'\big(\eta^{-1}(x)\big)\cdot\eta'\big(\eta^{-1}(y)\big)}
\end{align}
is a proper permuton density, with uniform marginals.
We now show that $f$ is $(\alpha,2\phi)$-standard with $\phi(x,y)=\alpha^{\alpha+1}(x+y)^{\alpha-1}$.

Since $h>0$ on $(0,1)^2$, so is $f$ and the corresponding permuton is corner-balanced.
Furthermore, by the symmetry of $h$ (and thus also of $f$), it suffices to check that $f$ is BL-smooth with parameters $\alpha$ and $\phi$.
To prove this, compute $\eta$ to find that, as $t\rightarrow0$, we have
\begin{align*}
    \eta(t)
    =\frac{2t}{\alpha}-\frac{1}{\alpha(\alpha+1)}\Big[1+t^{\alpha+1}-(1-t)^{\alpha+1}\Big]
    \sim\frac{t}{\alpha}\,.
\end{align*}
It follows that $\eta'(t)\sim1/\alpha$ which, combined with the definition of $f$ and $h$, leads to
\begin{align}\label{eq:gamma hf}
    f(tx,ty)
    \sim\frac{h(\alpha tx,\alpha ty)}{\alpha^{-1}\cdot\alpha^{-1}}
    \sim t^{\alpha-1}\alpha^{\alpha+1}(x+y)^{\alpha-1}\,.
\end{align}
This is exactly the definition of BL-smoothness, with the appropriate $\alpha$ and $\phi$.
Thus $f$ is $(\alpha,2\phi)$-standard, with $\alpha>0$ being arbitrary.
It is worth mentioning that there actually exists a lot more standard permutons with the same $\alpha$, for example by replacing $(x+y)^{\alpha-1}$ with $(ax+by)^{\alpha-1}$ in the definition of $h$, but we focused here on the existence of a single one and thus chose the simplest (and symmetric) case.

\subsection{Counter example}\label{sec:counter}

Since the smoothness assumption plays a key role in the limiting behaviour of the minimal inversion and is explored in details in Section~\ref{sec:bl}, we do not further explore the role of this assumption here.
Thus, in this section, and to conclude this work, we explore the role of the corner-balanced property via an example of a non standard permuton.

The role of having points in the top-left and bottom-right corner to prevent middle points from having too few inversions is visible in the statements of Lemma~\ref{lem:top-bottom} and Lemma~\ref{lem:left}.
However, upon working out the detailed statements for the necessary assumptions, another version of the corner-balanced propery was considered, which we refer to here as the \textit{weak corner-balanced} property, defined by the following statement.
For any $\epsilon>0$, there exists $0<\rho<1$ and $\delta>0$ such that
\begin{align*}
    \min\Big\{\mu\big([0,\epsilon]\times[\rho-\delta,1]\big),\mu\big([1-\epsilon,1]\times[0,\rho+\delta]\big)\Big\}
    >0\,.
\end{align*}
This weaker assumption allows for a slight extension of the value of $\rho$ and was meant to cover cases such as the permuton with uniform mass on $[0,1/2]^2\cup[1/2,1]^2$ (the empirical permuton of the identity on $\{1,2\}$).
Moreover, one might expect that the number of points within the band $[0,1]\times[\rho-\delta,\rho+\delta]$ is small enough that it eventually does not substantially affect the minimal inversion.
This expectation, however, is wrong, as shown by the following permuton.

Consider the permuton $\mu$ defined as follows.
First, $\mu$ gives mass $12/5$ uniformly to the two squares $[0,1/3]^2$ and $[2/3,1]^2$.
Then, it gives a mass on the rectangles $[0,1/3]\times[1/3,1/2]$, $[1/3,1/2]\times[0,1/3]$, $[1/2,2/3]\times[2/3,1]$, and $[2/3,1]\times[1/2,1/3]$ which decreases linearly going down from $[0,1/3]^2$ and $[2/3,1]^2$ towards $y=1/2$ and $x=1/2$.
Finally, it gives a mass to the portion of line $\{(x,x):1/3\leq x\leq2/3\}$ such that the marginals of $\mu$ are uniform.
In more formal terms, and by letting $f:[0,1]^2\rightarrow[0,\infty)$ be the function given by
\begin{align*}
    f(x,y)
    =\frac{12}{5}\left\{\begin{array}{ll}
        1 & \textrm{if $x,y\leq1/3$} \\
        3(1-2y) & \textrm{if $x\leq1/3$ and $1/3\leq y\leq1/2$} \\
        3(1-2x) & \textrm{if $y\leq1/3$ and $1/3\leq x\leq1/2$} \\
        0 & \textrm{otherwise}\,,
    \end{array}\right.
\end{align*}
then, for any $S\subseteq[0,1]^2$, $\mu$ satisfies
\begin{align}\label{eq:weird}
    \mu(S)
    =\int_Sf+\int_{S^{-1}}f+\int_{1/3}^{2/3}\left[1-\frac{12}{5}\big|1-2x\big|\right]\I{(x,x)\in S}dx\,,
\end{align}
where $S^{-1}=\{(1-x,1-y):(x,y)\in S\}$ is the symmetric of $S$ with respect to $y=1-x$.
We provide a visual representation of $\mu$ in Figure~\ref{fig:weird} and hope that the combination of these three definitions of $\mu$ help the reader understand its structure.

\begin{figure}[htb]
    \centering
    \includegraphics[width=\textwidth]{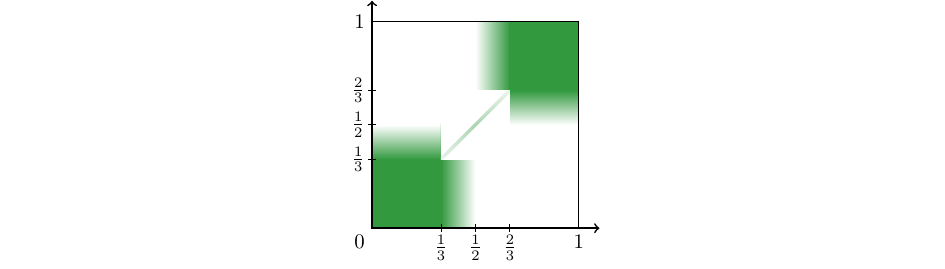}
    \caption{
    A heatmap of the permuton $\mu$ from~\eqref{eq:weird}.
    It has a density on $[0,1]^2\setminus[1/3,2/3]^2$ which is equal to $12/5$ on $[0,1/3]^2$ and on $[2/3,1]^2$, and then decays linearly in $x$ and $y$ to reach $0$ when $x=1/2$ or $y=1/2$.
    The value of $12/5$ is chosen so that its marginals are uniform on $[0,1/3]$ and on $[2/3,1]$.
    Past $1/3$, we complement the distribution with a mass on $\{(x,x):1/3\leq x\leq2/3\}$ evolving as a triangle function, going from its minimal value of $1/5$ when $x=1/3$ or $x=2/3$, to its maximal value of $1$ when $x=1/2$.
    Again, this choice is driven by the desire to have uniform marginals.
    }
    \label{fig:weird}
\end{figure}

While its definition might be intricate, the fact that $\mu$ is a permuton mostly follows from two facts.
First, for any $0\leq x\leq1/3$, $f$ is exactly defined so that its integral over $y\in[0,1]$ is equal to $1$.
Second, the function within the integral from $1/3$ to $2/3$ is exactly defined so that, for any $1/3\leq x\leq1/2$, it sums to $1$ when combined with the integral of $f$ over $y\in[0,1]$.
These two points make the $x$ marginals of $\mu$ uniform on $[0,1/2]$ and two symmetry arguments (one around $x=1/2$ and one around $y=1-x$) then lead to $\mu$ being a positive measure with uniform marginals, thus being a permuton.

It is clear from its definition that $\mu$ is BL and TR-smooth, both with parameters $\alpha$ and $\phi$ constant equal to $12/5$, but, it is not corner-balanced since $\mu([0,1/2]^2\cup[1/2,1]^2)=1$ (see Lemma~\ref{lem:tile ass}).
However, we also see that $\mu$ is weak corner-balanced, since $\mu([0,\epsilon]\times[1/2-\delta,1])>0$ and $\mu([1-\epsilon,1]\times[0,1/2+\delta])>0$ for any $\epsilon>0$ and $\delta>0$.
Thus, if the weak corner-balanced property were enough, then we should have that $\minv_n^\mu$ diverges, but this is not the case, as the minimal inversion in that case is actually $0$ with high-probability, as stated in the following proposition.

\begin{proposition}
    Let $\mu$ be the permuton from~\eqref{eq:weird}.
    Let $\minv_n^\mu$ be the minimal inversion of $n$ points distributed according to $\mu$, as defined in~\eqref{eq:main}.
    Then, the random variable $\minv_n^\mu$ converges to $0$ in distribution (and equivalently in probability).
\end{proposition}

\begin{proof}
    Fix $\delta=\delta_n=n^{-2/3}$ and assume from now on that $n$ is large enough so that $\delta<1/6$.
    Denote by $N_0$ the number of points within the union of $[0,1]\times[1/2-\delta,1/2+\delta]$ and $[1/2-\delta,1/2+\delta]\times[0,1]$ but outside of the diagonal $y=x$.
    Further denote by $N_d$ the number of points belonging to $\{(x,x):|x-1/2|\leq\delta\}$.
    We observe that, in the case where $N_d\geq1$ but $N_0=0$, then the points corresponding to $N_d$ do not form any inversion with other points from the set.
    Thus, we see that
    \begin{align*}
        \P{\minv_n^\mu=0}
        \geq\P{N_d\geq1,N_0=0}\,.
    \end{align*}
    Further observe that $N_0$ and $N_d$ are both binomial random variable of size $n$ and with respective parameters $p_0$ and $p_d$.
    To prove the desired result, it thus suffices to show that $np_0\rightarrow0$ and $np_d\rightarrow\infty$.

    For $p_0$, since we ignore the mass from the diagonal, we can completely use the symmetry of the distribution of $\mu$ with respect to $f$ and see that
    \begin{align*}
        p_0
        =4\int_{[1/2-\delta,1/2]\times[0,1/3]}f
        =\frac{48}{5}\int_{1/2-\delta}^{1/2}(1-2x)dx
        =\frac{48}{5}\delta^2\,.
    \end{align*}
    Recalling that $\delta=n^{-2/3}$, it follows that $np_0\rightarrow0$.
    On the other hand, for $p_d$, we use the symmetry of the measure and directly integrate the diagonal term to obtain that
    \begin{align*}
        p_d
        =2\int_{1/2-\delta}^{1/2}\left[1-\frac{12}{5}(1-2x)\right]dx
        =2\delta-\frac{24}{5}\delta^2\,.
    \end{align*}
    It is worth observing that $p_0+2p_d=4\delta$, as expected given the uniform marginals of $\mu$.
    Recalling again that $\delta=n^{-2/3}$, we see that $np_d\sim2n^{1/3}\rightarrow\infty$.
    This proves that $\minv_n^\mu=0$ with high probability, as desired.
\end{proof}

\section*{Acknowledgement}

The author would like to thank Sumit Mukherjee for insightful discussions on the topic, as well as the Centre International de Rencontres Math\'ematiques (CIRM) for the fostering environment they provided.
This project was made possible thanks to the support of the Department of Mathematics from the National University of Singapore.

\bibliographystyle{alpha}
\bibliography{main}

\end{document}